%% file: minmaxmin_discreteBudgeted.tex
\documentclass[preprint,12pt]{elsarticle}

\usepackage[utf8]{inputenc}
\usepackage{todonotes}
\usepackage{amsmath}
\usepackage{amsthm}
\usepackage{amssymb}
\usepackage{algpseudocode}
\usepackage{float}
\usepackage[ruled,linesnumbered]{algorithm2e}
\usepackage{tikz}
\usetikzlibrary{automata}
\usetikzlibrary{arrows,shapes}

\usepackage{graphicx}

\usepackage{amssymb,amsmath,verbatim,tabularx,enumitem}
\usepackage{colonequals,graphicx,psfrag}
\usepackage{caption, subcaption}

\include{newcommands}

\newtheorem{theorem}{Theorem}

\newtheorem{corollary}[theorem]{Corollary}

\newtheorem{observation}[theorem]{Observation}

\newcommand{\p}{y}
\newcommand{\px}{w}

\usepackage{multirow}

\begin{document}

\begin{frontmatter}

\title{Min-Max-Min Robustness for Combinatorial Problems with Discrete Budgeted Uncertainty}

\author{Marc Goerigk}
\address{Network and Data Science Management, University of Siegen, Germany}
\ead{marc.goerigk@uni-siegen.de}

\author{Jannis Kurtz}
\address{Department of Mathematics, RWTH Aachen University, Germany}
\ead{kurtz@mathc.rwth-aachen.de}

\author{Michael Poss}
\address{LIRMM, University of Montpellier, CNRS, France}
\ead{michael.poss@lirmm.fr}


\begin{abstract}
We consider robust combinatorial optimization problems with cost uncertainty where the decision maker can prepare $K$ solutions beforehand
and chooses the best of them once the true cost is revealed. Also known as min-max-min robustness (a special case of $K$-adaptability), it is a viable alternative to otherwise intractable two-stage problems. The uncertainty set assumed in this paper considers that in any scenario, at most $\Gamma$ of the components of the cost vectors will be higher than expected, which corresponds to the extreme points of the budgeted uncertainty set.

While the classical min-max problem with budgeted uncertainty is essentially as easy as the underlying deterministic problem, it turns out that the min-max-min problem is $\NPhard$ for many easy combinatorial optimization problems, and not approximable in general. We thus present an integer programming formulation for solving the problem through a row-and-column generation algorithm. While exact, this algorithm can only cope with small problems, so we present two additional heuristics leveraging the structure of budgeted uncertainty. We compare our row-and-column generation algorithm and our heuristics on knapsack and shortest path instances previously used in the scientific literature and find that the heuristics obtain good quality solutions in short computational times.
\end{abstract}


\end{frontmatter}

\baselineskip 20pt plus .3pt minus .1pt

\section{Introduction}

Let us consider a combinatorial optimization problem 
\begin{equation}\label{eq:min}
  \min_{x\in \X} \pmb{c}^\top  \pmb{x}, \tag{M$^1$}
\end{equation}
where $\X\subseteq \binvar{n}$ is the set of feasible solutions and $\pmb{c}\in\R^n$ is a cost vector. In many practical applications (e.g. uncertain road lengths for a shortest path problem or uncertain revenues in a project investment problem), the decisions $\pmb{x}$ must be taken prior to knowing the exact values of the cost vector $\pmb{c}$. In that context, one should account for these uncertainties when looking for a solution. One widely used approach is robust optimization, in which it is assumed that $\pmb{c}$ can take any value in a given uncertainty set $\cU\subseteq\R^n$, leading to the robust counterpart
\begin{equation}\label{eq:minmax}
  \min_{\pmb{x}\in \X}\max_{\pmb{c}\in \cU} \pmb{c}^\top  \pmb{x}. \tag{M$^2$}
\end{equation}
The min-max problem \eqref{eq:minmax} has been widely studied for discrete uncertainty sets $\Ud=\left\{ \pmb{c}^1, \ldots ,\pmb{c}^m\right\}$ (see the surveys \cite{aissi_minmax_survey,kasperski2016robust}). For most classical combinatorial problems the min-max problem \eqref{eq:minmax} is $\NPhard$ for general discrete uncertainty sets $\Ud$, even if $\Ud$ only contains two scenarios. What is more, discrete uncertainty sets rely on accurate historical data, which are often not available. When this is the case, it can be more natural for the decision maker to provide only two values $\{\hc_i, \hc_i+d_i\}$ with $d_i\geq 0$
for each uncertain cost component, representing the expected and worst case, respectively. Then, assuming that only $\Gamma$ uncertain parameters take their upper values in any scenario, Bertsimas and Sim~\cite{bertsimas2003combinatorial} introduced what is often called the discrete budgeted uncertainty set
\begin{equation*}
 \Ug = \left\{ \pmb{c}\in\R^n : c_i=\hc_i + \delta_i d_i, i\in [n],\; \pmb{\delta}\in\{0,1\}^n,\; \sum\limits_{i\in[n]} \delta_i \le \Gamma \right\}
\end{equation*}
where we use the notation $[n]:=\{1,\ldots,n\}$. The implicit description of $\Ug$ leads to min-max problems that are essentially as easy as the underlying deterministic problems (see \cite{bertsimas2003combinatorial}).

In some real-world situations, the min-max approach~\eqref{eq:minmax} can be too conservative, as no recourse action can be taken to remedy to the values taken by the uncertain parameters. A more flexible model has been introduced in~\cite{buchheim2016min} to overcome this limitation. The approach computes $K$ solutions and chooses the best of them for each realization of $\pmb{c}$ in $\cU$. Using discrete budgeted uncertainty in this setting leads to the min-max-min problem
 \begin{equation}
 \label{eq:minmaxmin}
 \min_{\x{1}, \ldots, \x{K} \in \X} \max_{\pmb{c}\in \Ug} \min_{k\in[K]} \pmb{c}^\top \x{k}. \tag{M$^3$}
 \end{equation}
 Approach~\eqref{eq:minmaxmin} is particularly well-suited in situations where one can prepare the ground before the uncertainty is revealed. It is a special case of the more general $K$-adaptability approach from~\cite{hanasusanto2015k}, where additional first-stage costs are allowed. Examples are numerous, e.g. transporting relief supplies or evacuating citizens in case an uncertain disaster arises~\cite{chang2007scenario,liberatore2013uncertainty}, hub locations~\cite{alumur2012hub}, or parcel deliveries~\cite{eufinger2017robust,subramanyam2017k}; see also~\cite{ChasseinGKP19} for more details.
 
It has been shown in~\cite{buchheim2016min} that Problem~\eqref{eq:minmaxmin} is as easy as the underlying problem \eqref{eq:min} if $\cU$ is a convex uncertainty set and $K\ge n+1$. For discrete uncertainty however, the min-max-min problem is at least as complex as the min-max problem which is, more often than not, $\NPhard$~\cite{JannisDiscrete}. Regarding the budgeted uncertainty, its convex hull
\begin{equation}
\label{eq:convU}
 \left\{ \pmb{c}\in\R^n : c_i\in[\hc_i, \hc_i+d_i], i\in [n],\; \sum\limits_{i\in[n]} \frac{c_i-\hc_i}{d_i} \le \Gamma \right\}
\end{equation}
has been considered in~\cite{Chassein2017} and~\cite{ChasseinGKP19}, where the authors study the theoretical complexity of the problem, and propose efficient solution algorithms, respectively.

The focus of this paper is Problem~\eqref{eq:minmaxmin} under the discrete budgeted uncertainty set $\Ug$. First, we remark that, unlike the min-max problem, it is not equivalent to replace $\Ug$ by its convex hull \eqref{eq:convU}, see for instance~\cite{JannisDiscrete} for an example, so that the results for convex uncertainty sets proved in~\cite{buchheim2016min,Chassein2017} do not carry over the problem studied in this paper. Moreover the discrete version of the budgeted uncertainty set is inevitable if we want to model edge failures for problems on graphs or adversarial attacks on labeled data in classification problems, see \cite{bertsimas2018robust} for the latter application.

We prove in Section~\ref{sec:negativeresults} that, unlike the classical min-max problem, Problem~\eqref{eq:minmaxmin} for discrete budgeted uncertainty is $\NPhard$ for the shortest path problem, the spanning tree problem, the assignment problem, the knapsack problem, the selection problem and even for the unconstrained minimization problem. We also show that, in general, Problem~\eqref{eq:minmaxmin} cannot be approximated in polynomial time. On the positive side, we provide dynamic programming algorithms for the unconstrained and the knapsack problem, which are essentially of theoretical interest. Section~\ref{sec:ilp} turns to integer linear programming (ILP) formulations. We first show that computing the inner maximization problem 
\[\max_{\pmb{c}\in \Ug} \min_{k\in[K]} \pmb{c}^\top \x{k}\]
is $\NPhard$, suggesting that no compact  ILP may be available for~\eqref{eq:minmaxmin}. We then introduce an assignment-based formulation, that is embedded into a row-and-column generation algorithm. We pursue numerical approaches for the problem in Section~\ref{sec:heur} where we provide two heuristic algorithms. Leveraging known results for budgeted uncertainty, we can prove both algorithms solve only polynomially many deterministic problems. Our ILP formulation and the heuristics are numerically assessed in Section~\ref{sec:xp} on instances of knapsack and shortest path problems previously considered in the literature. Section~\ref{sec:conclusion} concludes the paper.

\section{Complexity results}
\label{sec:negativeresults}

In this section, we provide complexity results for the min-max-min problem \eqref{eq:minmaxmin} for 
the unconstrained binary problem (where $\X = \{0,1\}^n$),
the selection problem (where $\X = \{ \pmb{x}\in\{0,1\}^n : \sum_{i\in[n]} x_i = p\}$ for an integer $p<n$),
the knapsack problem,
the spanning tree problem, the assignment problem, 
and the shortest path problem. Note that for the knapsack problem, we consider a max-min-max problem instead of~\eqref{eq:minmaxmin}.
It turns out that all of the mentioned problems are at least weakly $\NPhard$, and under some configurations even strongly $\NPhard$. Table \ref{tbl:compl_min-max-min_discr} summarizes the results which we prove in the following. While the unconstrained binary problem and the selection problem are trivial problems in the nominal case, they are often considered in the robust optimization literature (see, e.g., \cite{buchheimrobust,chassein2018recoverable}). We further show that the shortest path problem becomes inapproximable, and provide a dynamic programming algorithm for the knapsack problem.

\begin{table}[htb]
\centering
\begin{tabular}{l|ll|l}
Problem   & $K$ fixed & $K$ input & Reference\\ 
\hline 
Unconstrained & weakly $\NP$-hard & strongly $\NP$-hard & Theorem~\ref{thm:uncons}\\
Selection & weakly $\NP$-hard & strongly $\NP$-hard & Theorem~\ref{thm:selection}\\
Knapsack & weakly $\NP$-hard & strongly $\NP$-hard & Theorem~\ref{thm:hard-others}\\
Spanning Tree & $\NP$-hard & strongly $\NP$-hard & Theorem~\ref{thm:hard-others}\\
Assignment & $\NP$-hard & strongly $\NP$-hard & Theorem~\ref{thm:hard-others}\\
Shortest Path & strongly $\NP$-hard & strongly $\NP$-hard & Theorem~\ref{thm:shortestpath}\\
\end{tabular}
\caption{Complexity of Problem \eqref{eq:minmaxmin} for $\Ug$.}\label{tbl:compl_min-max-min_discr}
\end{table} 

\subsection{$\NP$-hard cases and inapproximability}\label{sec:nphard}

In the following we prove $\NP$-hardness results for Problem \eqref{eq:minmaxmin} for several combinatorial problems.

\begin{theorem}
\label{thm:uncons}
Problem \eqref{eq:minmaxmin} for the unconstrained binary problem is weakly $\NPhard$, 
even if $K=2$. It is strongly $\NPhard$ 
when $K$ is part of the input.
\end{theorem}
\begin{proof}
First assume that $\Gamma=1$ and $K=2$ and consider a set of integers $a_i\in\N$, $i\in [n]$. The partition problem, i.e. deciding whether there exists a subset $S\subseteq [n]$ such that $\sum_{i\in S}a_i=\sum_{i\in [n] \backslash S}a_i$, is known to be $\NPhard$. Next, consider an instance of problem~\eqref{eq:minmaxmin}
where $\hat{c}_i=-a_i$ and $d_i=M$ for each $i\in [n]$ and $M=\sum_{i\in [n]}a_i$. By the choice of $M$ in each optimal solution it holds $\x{1}_i\neq \x{2}_i$ for all $i\in[n]$ and the large deviation $M$ will affect the solution with the smaller nominal costs. The optimal value is therefore larger or equal to $-\frac{1}{2}M$. Then, the partition instance is a yes instance if and only if the min-max-min problem has an optimal value equal to $-\frac{1}{2}M$. 

Since the unconstrained binary problem is a special case of the knapsack problem the pseudo-polynomial algorithm presented in Section \ref{sec:dp} can be applied to \eqref{eq:minmaxmin} for the unconstrained problem, which proves the weak $\NPhard$ness.

The proof above extends to the case when $K$ is part of the input by setting $\Gamma=K-1$. A similar reasoning shows that one can decide if a partition into $K$ sets  having the same costs exists if and only if the corresponding min-max-min problem has an optimal value equal to $-\frac{1}{K}M$. 
The problem generalizes the 3-partition problem (see below), which is $\NPhard$ in the strong sense, which proves the result.
\end{proof}

\begin{theorem}
\label{thm:selection}
The problem \eqref{eq:minmaxmin} is weakly $\NP$-hard for the selection problem, even if $K=2$. It is strongly $\NP$-hard if $K$ is part of the input.
\end{theorem}
\begin{proof}

For $K=2$ we can use a similar construction as in the proof of Theorem~\ref{thm:uncons} with $p=n/2$, as the partition problems remains (weakly) $\NPhard$ if both partitions are required to have equal size. Note that negative cost vectors are not required in this case. Moreover, one can adapt the pseudo-polynomial algorithm presented in Section \ref{sec:dp} to the selection problem which proves the weakly $\NPhard$ness.

To prove strong $\NP$-hardness for the case that $K$ is part of the input we reduce the $3$-partition problem to Problem \eqref{eq:minmaxmin}. The $3$-partition problem is defined as follows: For given values $c_1,\ldots ,c_{3m}\in\N$ and a bound $B\in\N$ such that $\frac{B}{4} < c_i < \frac{B}{2}$ for all $i=1,\ldots ,3m$ and $\sum_{i=1}^{3m} c_i = mB$ we have to decide if there exist subsets $A_1,\ldots ,A_m\subseteq\left\{ 1,\ldots ,3m\right\}$ such that $\sum_{i\in A_j} c_i = B$ for all $j=1\ldots ,m$. Note that each set $A_j$ must contain exactly $3$ elements. The $3$-partition problem is known to be strongly $\NP$-hard~\cite{gareyjohnson}.

We assume we have a given instance of the $3$-partition problem as above. We now define an instance of the min-max-min selection problem in dimension $3m$ where $p=3$, $K=m$, $\Gamma = K-1$ and we define $\Ug$ with $\hat c_i = c_i$ and $d_i = M$ where $M = \sum_{i=1}^{3m} c_i$. By the definition of $M$ and $\Gamma$ the $K$ solutions define a partition of the ground set $\left\{ 1,\ldots ,3m\right\}$ in an optimal solution of~\eqref{eq:minmaxmin}, and each solution uses exactly $3$ elements. The given $3$-partition is then a yes instance if and only if Problem \eqref{eq:minmaxmin} has an optimal value lower or equal to $B$.

\end{proof}

\begin{theorem}\label{thm:hard-others}
Even if $K=2$, Problem \eqref{eq:minmaxmin} for the spanning tree problem and the assignment problem is $\NPhard$; for the knapsack problem it is weakly $\NPhard$.
It is strongly $\NPhard$ for all the mentioned problems if
$K$ is part of the input.
\end{theorem}

The proof of Theorem~\ref{thm:hard-others} follows the line of the proof of Theorem~\ref{thm:uncons} and is provided in~\ref{app:proofs}. The result for the knapsack problem directly follows from Theorem~\ref{thm:uncons} by choosing knapsack weights and a knapsack capacity such that all vectors in $\binvar{n}$ are feasible.

\begin{theorem}
\label{thm:shortestpath}
Problem \eqref{eq:minmaxmin} for the shortest path problem is strongly $\NPhard$, 
even if $K=2$.
\end{theorem}
\begin{proof}
Assume that $\Gamma=1$ and $K=2$ and consider a weighted undirected graph $G=(V,E,\pmb{w})$ and two nodes $s$ and $t$ in $V$. It is known that finding two edge-disjoint paths between $s$ and $t$ such that the length of the longer path is minimized is $\NPhard$ in the strong sense \cite{li1990complexity}. Next, consider an instance of problem \eqref{eq:minmaxmin} where we need to find paths between $s$ and $t$ in $G$ and define $\hat{c}_e=w_e$ and $d_e=M$ for each $e\in E$, where $M$ can be chosen as $M=\sum_{e\in E}\hat{c}_e$. Clearly there exist two disjoint paths in $G$ where the longer path has costs lower or equal to some value $L$
if and only if the latter instance of the min-max-min problem has an optimal value lower or equal to $L$.

\end{proof}

Our last complexity result shows that the problem may not be approximable even if the underlying problem is polynomially solvable.
\begin{theorem}
\label{thm:shortestpath-inappr}
Problem \eqref{eq:minmaxmin} for the shortest path problem 
where paths may have up to $5$ edges
cannot be approximated in polynomial time if $K$ is part of the input unless $\P=\NP$.
\end{theorem}
\begin{proof}
Consider a weighted undirected graph $G=(V,E,\pmb{w})$ and two nodes $s$ and $t$ in $V$. It is known that determining whether there exist $K$ edge-disjoint paths between $s$ and $t$ having at most $5$ edges is $\NPhard$~\cite{ItaiPS82}. By defining $\Gamma=K-1$, $\hat{\pmb{c}}=\pmb{0}$ and $\pmb{d}=\pmb{1}$, the optimal solution cost of~\eqref{eq:minmaxmin} is $0$ if and only if the answer to the decision problem is yes.
\end{proof}

\subsection{Dynamic programming for the knapsack problem}\label{sec:dp}
We provide a dynamic programming algorithm for the knapsack version of problem \eqref{eq:minmaxmin}, which runs in pseudo-polynomial time. As the complexity of the algorithm is  high, its purpose is essentially theoretical. In what follows we suppose the feasibility set is given by
$$
\X_{KP} = \left\{\pmb{x}\in \{0,1\}^n : \sum_{i\in[n]} w_i x_i \leq C\right\}
$$
where $w_i$ denotes the weight of item $i$ and $C$ the total available capacity. Our algorithm runs in pseudo-polynomial time if $K$ and $\Gamma$ are fixed, and has an exponential running-time in general, in accordance with the complexity stated in Theorem~\ref{thm:uncons}. The algorithm can easily be adapted to handle the unconstrained binary problem by choosing knapsack weights and a capacity such that all binary vectors are feasible and transforming the problem into a minimization problem. For the selection problem, the dynamic programming algorithm extends by enforcing that the capacity constraint be satisfied at equality when computing the costs in step~\ref{step:cost} of Algorithm~\ref{algo:dpk2}.

\begin{algorithm}[htb]
\small
\caption{Solving the robust knapsack problem for $K=2$ and $\Gamma=1$}
\label{algo:dpk2}
\algorithmicrequire \ $\X_{KP}$, $\Ug$, $K=2$, $\Gamma=1$\\
   $\S=\{(0,0,0,0,0,0,0)\}$\\
  \For{$i \in [n]$}{
    \For{$(w^{(1)},w^{(2)},c^{(1)},c^{(2)},d^{(1)},d^{(2)},d^{(1,2)}) \in \S$}{
    \If{$w^{(1)}+w_i\leq C$}{
      $\pmb{s}^{(1)}=(w^{(1)}+w_i,w^{(2)},c^{(1)}+\hc_i,c^{(2)},\max\left(d^{(1)}, d_i\right),d^{(2)},d^{(1,2)})$ \label{step:x1}\\
      \If{$w^{(2)}+w_i\leq C$}{
	$\pmb{s}^{(1,2)}=(w^{(1)}+w_i,w^{(2)}+w_i,c^{(1)}+\hc_i,c^{(2)}+\hc_i,\max\left(d^{(1)}, d_i\right),\newline \max\left(d^{(2)}, d_i\right),\max\left(d^{(1,2)}, d_i\right))$ \label{step:x12}
      }}\ElseIf{$w^{(2)}+w_i\leq C$}
      {
	$\pmb{s}^{(2)}=(w^{(1)},w^{(2)}+w_i,c^{(1)},c^{(2)}+\hc_i,d^{(1)},\max\left(d^{(2)}, d_i\right),d^{(1,2)})$ \label{step:x2}
	}
      
    $\S\leftarrow \S \cup \{\pmb{s}^{(1)},\pmb{s}^{(2)},\pmb{s}^{(1,2)}\}$
    }
  }
	Compute $\pmb{s}_{\max}\in \S$ with maximal $cost(\pmb{s})$ (see Eq.~\eqref{eq:cost}) \label{step:cost}

  \algorithmicensure \ $\pmb{s}_{\max}$
\end{algorithm}

Algorithm~\ref{algo:dpk2} follows the classical dynamic programming algorithm for scheduling jobs on unrelated machines proposed in~\cite{Horowitz:1976} by iterating over all items and creating labels $\pmb{s}\in \S$ for each possible variant of adding the item to the feasible solutions or not. For ease of notation 
we present the idea of the algorithm for the case $K=2$ and $\Gamma=1$, and sketch the generalization in~\ref{app:DP}.

For each item $i\in[n]$ we assume that the current set of labels $\S$ has been constructed by deciding whether the first $i-1$ items are added to (partial) solution $\x{1}\in \X_{KP}$, $\x{2}\in \X_{KP}$ or to both. Hence, the partial solutions correspond to sets of items $S^{(1)}\subseteq [i-1]$ and $S^{(2)}\subseteq [i-1]$, respectively. As indexing the label $\pmb{s}$ by $S^{(1)}$ and $S^{(2)}$ could possibly lead to exponentially many labels, we instead define the label as
\[
\pmb{s}=(w^{(1)},w^{(2)},c^{(1)},c^{(2)},d^{(1)},d^{(2)},d^{(1,2)}),
\]
where $w^{(k)}=\sum_{i \in S^{(k)}}w_i$, $c^{(k)}=\sum_{i \in S^{(k)}}c_i$, and $d^{(k)}=\max_{i \in S^{(k)}}d_i$ for $k\in\{1,2\}$. The value $d^{(1,2)}$ represents the highest deviations among the items included in both solutions up to now, that is, $d^{(1,2)}=\max_{i \in S^{(1)}\cap S^{(2)}}d_i$. Next, we have four possibilities concerning the addition of item $i$ to the partial solutions:
\begin{enumerate}
 \item $S^{(1)}\leftarrow S^{(1)}\cup\{i\}$, yielding the new label $\pmb{s}^{(1)}$, see step~\ref{step:x1}
 \item $S^{(2)}\leftarrow S^{(2)}\cup\{i\}$, yielding the new label $\pmb{s}^{(2)}$, see step~\ref{step:x2}
 \item $S^{(1)}\leftarrow S^{(1)}\cup\{i\}$ and $S^{(2)}\leftarrow S^{(2)}\cup\{i\}$, yielding the new label $\pmb{s}^{(1,2)}$, see step~\ref{step:x12}.
 \item item $i$ is added to none of the two solutions, which does not change the current label.
\end{enumerate}
Notice that the first three possibilities occur only if the resulting partial solutions do not exceed the capacity of the knapsack.

Finally, we detail how to calculate the costs of a label, i.e. the worst-case over all scenarios in $\Ug$ for the corresponding solution. Note that since $\Gamma=1$ that cost can be computed by examining the following three possibilities: either we allow a deviation on the item with the largest deviation $d_i$ contained in the first solution or contained in the second solution or in both solutions. Therefore the worst-case costs can be calculated by
\begin{equation}\label{eq:cost}
\begin{aligned}
cost(\pmb{s})=\max\Big\{ &\min(c^{(1)}+d^{(1)},c^{(2)}),\\
&\min(c^{(1)}+d^{(1,2)},c^{(2)}+d^{(1,2)}), \\
&\min(c^{(1)},c^{(2)}+d^{(2)}) \Big\}.
\end{aligned}
\end{equation}
To investigate the run-time of the algorithm consider $\maxc=\max_{i\in[n]} \hat{c}_i$, $\maxd=\max_{i\in[n]}d_i$ and $\overline{w} = \max_{i\in[n]} w_i$. Then the largest value we have to consider for $c^{(1)},c^{(2)}$ is $n\maxc$, the largest value for $w^{(1)},w^{(2)}$ is $n\maxw$ and the largest value for $d^{(1)},d^{(2)},d^{(1,2)}$ is $\maxd$. Therefore the number of different labels we have to consider in the algorithm is at most $n^4\maxc^2\overline{w}^2\maxd^3$, so that the running time of Algorithm~\ref{algo:dpk2} is in $O(n^5\maxc^2\overline{w}^2\maxd^3)$. 
Since we may record the indices that deviate instead of the deviations themselves, another valid bound for the running time of the algorithm is $O(n^8\maxc^2\overline{w}^2)$. Specifically, we could define any label in $\S$ as the 7-tuple $\pmb{s}'=(w^{(1)},w^{(2)},c^{(1)},c^{(2)},i^{(1)},i^{(2)},i^{(1,2)})$ where $i^{(1)},i^{(2)}$ and $i^{(1,2)}$ belong to $[n]$. 

Many dynamic programming algorithms lead to approximation algorithms that can provide $(1+\epsilon)$-approximate solutions with a complexity that is polynomial in the input of the problem and $1/\epsilon$. We prove below that this is not the case here, by adapting the reduction from the equipartition problem to the knapsack problem from \cite{kellerer2003knapsack}.
\begin{theorem}
  There is no FPTAS for the knapsack variant of problem \eqref{eq:minmaxmin} unless $\P=\NP$, even in the case $\Gamma=1$.
 \end{theorem}
 \begin{proof}
We reduce the equipartition problem to the decision version of \eqref{eq:minmaxmin} for $\Gamma=1$ and $K = 2$. Given an instance of the equipartition problem i.e. $a_i\in\N$ for each $i\in N=[n]$ where $n$ is even, the equipartition problem is to decide if there exists an index set $I\subset N$ with $|I|=\frac{n}{2}$ such that $\sum_{i\in I}a_i=\sum_{i\in N\setminus I}a_i$. This problem is known to be $\NPhard$ \cite{gareyjohnson} . We define a knapsack instance as follows: for each $i\in N$ we set $w_i=a_i$, $\hat{c}_i=1$ and $d_i=-n$. The capacity is set to $C=\frac{1}{2}\sum_{i\in N}a_i$. Clearly, there exists an equipartition of the elements of $N$ if and only if there exists a solution of Problem \eqref{eq:minmaxmin} for the knapsack instance with profit at least $\frac{n}{2}$. This yields the result, as an FPTAS would compute an $\frac{1}{n+1}$-approximate solution in polynomial time, which would be optimal for that instance.
\end{proof}

Note that while it is not possible to construct an FPTAS for the robust knapsack problem, it stands to reason that the dynamic programming approach presented in this section can be used to construct a PTAS based on cost inflation.

\section{Exact algorithm and lower bounds}
\label{sec:ilp}

We first prove that even evaluating the objective function of Problem~\eqref{eq:minmaxmin}, i.e., calculating
\[
cost(\xx):=\max_{\pmb{c}\in \Ug}\min_{k\in[K]} \pmb{c}^\top \x{k}
\]
for fixed $\xx=(\x{1},\ldots,\x{K})$ is strongly $\NP$-hard. This makes it unlikely that a compact ILP formulation for Problem~\eqref{eq:minmaxmin} exists.
\begin{theorem}
 Evaluating the objective function of Problem \eqref{eq:minmaxmin} for a given solution is strongly $\NP$-hard.
\end{theorem}
\begin{proof}
 Given an integer $N$ and a collection $\mathcal{S}$ of $m$ sets $S_i\subseteq [N]$, the set cover problem looks for a sub-collection of $\mathcal{S}$ of cardinality not greater than $L$ and whose union equals $[N]$. We construct a reduction through the following instance of $cost(\xx)$. We set $K=N$, $\Gamma=L$, $n=m$, $\hat{c}_i=0$ and $d_i=1$ for each $i\in[n]$. Further, we define $x^{(k)}_i=1$ iff $k\in S_i$. Clearly, $cost(\xx)\geq 0$. We prove next that $cost(\xx)\geq 1$ if and only if the answer to the set cover instance is \emph{yes}.
 
 Let us define $\mathcal{Z}=\{\pmb{z}\in\{0,1\}^n:\sum_{i\in[n]} z_i\leq \Gamma\}$ as the set of vectors $\pmb{z}$ describing costs $\pmb{c}=(\hc_i+d_i z_i)_{i\in[n]}\in \Ug$. There exists a bijection between the elements of $\mathcal{Z}$ and the sub-collections of $\mathcal{S}$ of cardinality not greater than $\Gamma$. Further, we see that $\pmb{c}^\top \pmb{x}^{(k)}=\sum_{i\in [n]} z_i x^{(k)}_i$, so setting $z_i=1$ implies that $\sum_{i\in [n]} z_i x^{(k)}_i \geq 1$ for each $k\in S_i$. From the definition of $cost(\xx)$, we have that
 \begin{equation}
 \label{eq:costred}
 cost(\xx)\geq 1\Leftrightarrow \exists \pmb{z}\in \mathcal{Z}, \forall k\in [K]: \sum_{i\in [n]} z_i x^{(k)}_i \geq 1.
 \end{equation}
 If the answer to the set cover problem is \emph{yes} and is provided by the sub-collection indexed by $\mathcal{M}$, then we set $z_i=1$ for each $i\in \mathcal{M}$ and obtain from \eqref{eq:costred} that $cost(\xx)\geq 1$. If the answer is \emph{no}, then for each $\pmb{z}\in \mathcal{Z}$, there exists $k\in [K]$ such that $\sum_{i\in [n]} z_i x^{(k)}_i = 0$, and we obtain $cost(\xx) = 0$.
\end{proof}

In the following we provide an ILP formulation for Problem \eqref{eq:minmaxmin} and derive a row-and-column generation algorithm based on this formulation. Consider first an arbitrary discrete uncertainty set $\cU=\left\{ \pmb{c}^1,\ldots , \pmb{c}^m\right\}$ and let binary variables $\p_{kj}$ define an assignment between the scenarios and the solutions, where $\p_{kj}=1$ if and only if $\x{k}$ has the minimal objective value over all $\x{1},\ldots ,\x{K}$ in scenario $\pmb{c}^j$. We can reformulate Problem \eqref{eq:minmaxmin} as
\begin{equation}\label{eq:assignmentformulation}\tag{Master}
\begin{aligned}
\min \quad&\ \omega \\
 \text{s.t.} \quad & \ \omega\geq \sum_{k\in[K]}\p_{kj} \left( \sum_{i\in[n]} c^j_i x^{(k)}_i\right) & \forall j \in [m] \\
 & \sum_{k\in[K]}\p_{kj} = 1 & \forall j \in [m] \\
 & \p_{kj}\in \{0,1\} & \forall k\in[K],\ j \in [m]\\
 & \x{k}\in \X & \forall k \in [K],
\end{aligned}
\end{equation}
where the product $\p_{kj}x^{(k)}_i$ can be linearized introducing the additional variables $\px_{kji}=\p_{kj}x^{(k)}_i$ and rewriting~\eqref{eq:assignmentformulation} to
\begin{equation}
\begin{aligned}
\min \quad&\ \omega \\
 \text{s.t.} \quad & \ \omega\geq \sum_{k\in[K]}\sum_{i\in[n]} c^j_i\px_{kji} & \forall j \in [m] \\
 & \sum_{k\in[K]}\p_{kj} = 1 & \forall j \in [m] \\
 & \px_{kji} \geq x^{(k)}_i+\p_{kj}-1& \forall k \in [K], j\in [m], i \in [n] \\
 & \p_{kj}\in \{0,1\} & \forall k\in[K],\ j \in [m]\\
 & \px_{kji}\ge 0 & \forall k \in [K], j\in [m], i \in [n] \\
 & \x{k}\in \X & \forall k \in [K].
\end{aligned}
\end{equation}
Note that \eqref{eq:assignmentformulation} has exponentially many variables and constraints in case of discrete budgeted uncertainty. The first ingredient of our approach, described in Algorithm~\ref{alg:generatescenarios}, is to solve \eqref{eq:assignmentformulation} for a starting set $\cU'\subset \Ug$ and to iteratively add a new scenario which is the optimal solution of problem
\begin{equation}\label{eq:dualprobalgorithm}\tag{Slave}
\max_{\pmb{c}\in\Ug} \min_{k\in[K]} \pmb{c}^\top\x{k}
\end{equation}
for the current solution $\xx$ to the restricted master problem.
Note that the optimal value of the latter problem is the objective value $cost(\xx)$ of $\x{1},\ldots ,\x{K}$ for problem \eqref{eq:minmaxmin}, which can be computed through the following IP formulation:
\begin{equation}\label{eq:IPformulationcosts}
\begin{aligned}
\max & \quad z \\
s.t. \quad & z\le \hat c^\top \x{k} + \sum_{i=1}^{n} \delta_id_i\x{k}_i \ k\in [K] \\
& \sum_{i=1}^{n} \delta_i\le \Gamma \\
& \delta_i\in\left\{ 0,1\right\} \ i\in [n] .
\end{aligned}
\end{equation}
Clearly the optimal value of \eqref{eq:assignmentformulation} for a subset of scenarios is a lower bound for Problem \eqref{eq:minmaxmin} while the optimal value of \eqref{eq:dualprobalgorithm} is an upper bound. Therefore Algorithm \ref{alg:generatescenarios} iteratively calculates upper and lower bounds with decreasing gap. A similar idea for robust two-stage problems was already presented in \cite{zeng2013solving}.

Algorithm \ref{alg:generatescenarios} calculates an optimal solution of Problem \eqref{eq:minmaxmin}. Since there is a finite number of feasible solutions, we can only generate a finite number of scenarios in the loop and therefore the algorithm terminates in a finite number of steps.
\begin{algorithm}[]
\caption{~~Row-and-column generation algorithm}
\label{alg:generatescenarios}                           
\algorithmicrequire \ $\X$, $\Ug$, $K$\\
$\cU' \leftarrow \emptyset$ \\
Choose any $\pmb{c}\in \Ug$ \\
\Repeat{$LB = UB$}{
$\cU' \leftarrow \cU'\cup\{\pmb{c}\}$ \\
Solve~\eqref{eq:assignmentformulation} with respect to $\cU'$ \\
Set $\xx$ as its optimal solution, $LB$ as its objective value \\
Solve~\eqref{eq:dualprobalgorithm} with respect to $\xx$ \\
Set $\pmb{c}$ as its optimal solution, $UB$ as its objective value}
\algorithmicensure \ $\xx$
\end{algorithm}

As mentioned above, Algorithm \ref{alg:generatescenarios} iteratively calculates a non-decreasing sequence of lower bounds for Problem \eqref{eq:minmaxmin}. Nevertheless as our computational experiments show these lower bounds are hard to compute and tend to be far from the optimal value even after one hour of computation time; see Section \ref{sec:xp}. In the following we present a different lower bound for Problem \eqref{eq:minmaxmin} which turns out to be tighter as well as easier to compute.
\begin{observation}
The optimal value of problem
\begin{equation}\label{eq:max-minLB}\tag{MMLB}
\max_{\pmb{c}\in \Ug}\min_{\pmb{x}\in \X} \pmb{c}^\top \pmb{x}
\end{equation}
is a lower bound for Problem \eqref{eq:minmaxmin}.
\end{observation}
\begin{proof}
Clearly for every solution $\x{1},\ldots ,\x{K}$ we have
\[
\max_{\pmb{c}\in \Ug}\min_{\pmb{x}\in \X} \pmb{c}^\top \pmb{x} \le \max_{\pmb{c}\in \Ug}\min_{k\in [K]} \pmb{c}^\top \pmb{x}^{(k)}
\]
which proves the result.
\end{proof}
The problem \eqref{eq:max-minLB} can be solved by a classical row-generation method as follows:
For a subset of solutions $\X'\subset \X$ calculate an optimal solution $(\pmb{c}^*,z^*)$ of problem
\begin{align*}
\max\ & z \\
\text{s.t. } & z\le \pmb{c}^\top \pmb{x} & \forall \pmb{x}\in \X' \\
 & \pmb{c}\in \Ug 
\end{align*}
and afterwards solve the deterministic problem 
\[
\min_{\pmb{x}\in\X} \ (\pmb{c}^*)^\top \pmb{x} .
\]
Add the optimal solution of the latter problem to $\X'$ and iterate. Stop if the latter optimal value is larger than or equal to $z^*$.


\section{Heuristic algorithms}
\label{sec:heur}
In this section we present two heuristic algorithms which are based on the idea to find a partition of the uncertainty set into $K$ subsets and calculate the optimal min-max solution for each of the subsets, see the general scheme presented in Algorithm~\ref{alg:heuristic}. To end up with a fast algorithm the min-max problem for each subset should be computationally tractable. For both heuristics we derive a $K$-partition of the budgeted uncertainty set such that each subset remains a budgeted uncertainty set or has a structure which is close to a budgeted uncertainty set. In both cases we can show that each of the min-max problems in Algorithm \ref{alg:heuristic} can be solved by solving a polynomial number of deterministic problems \eqref{eq:min}.
\begin{algorithm}[htb]
\caption{Heuristic Algorithm for Problem \eqref{eq:minmaxmin} with $K\le n$.}\label{alg:heuristic}
\algorithmicrequire \ $\X$, $\Ug$, $K$\\
  Calculate a partition $\U_1\cup \ldots \cup \U_K = \Ug$. \\
  Calculate \label{step:minmaxheuristic}
  \[
  \x{k} = \argmin_{\pmb{x}\in \X} \max_{\pmb{c}\in\U_k} \pmb{c}^\top \pmb{x} \qquad \forall k\in[K].
  \]

  \algorithmicensure \ $\x{1},\ldots ,\x{K}$
\end{algorithm}

\subsection{Heuristic 1}
In this section we derive a partition of the budgeted uncertainty set $\Ug$ such that each of the subsets has a similar structure as the classical budgeted uncertainty set.  Furthermore each subset of the partition covers scenarios which are close to each other in the sense that there exists a solution $\pmb{x}\in \X$ which works well for most of the scenarios. We show that for each subset of the partition the min-max problem can be solved by solving a polynomial number of deterministic problems.

We partition the budgeted uncertainty set such that in each subset $\cU_k$ a subsequence of items is selected and in each scenario of the subset at least one of the items deviates from its mean value. More precisely, let any ordering of the indices $i_1,\ldots ,i_n$ be given. Without loss of generality, we assume $i_\ell=\ell$ in the following presentation. For $t:=\lfloor \frac{n}{K}\rfloor$, we define 
\begin{align}
\cU_k:=\Bigg\{ \pmb{c}\in \R^n \ : \ &c_i=\hat c_i + \delta_i d_i, \ \pmb{\delta}\in \{0,1\}^n, \label{eq:partitionheuristic}\\
&\sum_{i=1}^{(k-1)t}\delta_{i}=0, \sum_{i=(k-1)t+1}^{kt}\delta_{i}\ge 1,\sum_{i=(k-1)t + 1}^{n} \delta_{i} \le \Gamma \Bigg\} \nonumber
\end{align}
for each $k=1,\ldots , K-1$ and 
\begin{align*}
\cU_K:=\Bigg\{ \pmb{c}\in \R^n \ : \ &c_i=\hat c_i + \delta_i d_i, \ \pmb{\delta}\in \{0,1\}^n, \\
&\sum_{i=1}^{(K-1)t}\delta_{i}=0, \sum_{i=(K-1)t+1}^{n}\delta_{i}\ge 1, \sum_{i=(K-1)t + 1}^{n} \delta_{i} \le \Gamma \Bigg\} .
\end{align*}
Note that the only difference in the definition of $\cU_K$ is that the second sum instead of containing $t$ summands, additionally contains all summands which are left due to rounding of $t$. For ease of notation we do not consider this special case in the following.

It is easy to see that $\U_1\cup \dots \cup \U_K= \Ug\setminus\left\{\hat c\right\}$. Furthermore all of the subsets $\cU_k$ have a budgeted-like structure. We will use this structure in the following theorem to show that the classical min-max problem in Step \ref{step:minmaxheuristic} of Algorithm \ref{alg:heuristic} can be solved in polynomial time if an oracle for the underlying deterministic problem is given. We define in the following $(x)_+:=\max\left\{ x,0\right\}$. The following result is related to Theorem~3 from \cite{bertsimas2003combinatorial} and its generalizations in~\cite{Poss17}.
\begin{theorem}\label{thm:polyTimePartitionProblems}
The min-max problem with uncertainty set $\cU_k$ can be solved by solving the deterministic problems
\begin{equation}\label{eq:dualMinMaxFinal}
\alpha^*\Gamma - \beta^* + t(\beta^*-\alpha^*)_+  + \min_{\pmb{x}\in \X} \hat{\pmb{c}}^\top \pmb{x} + \pmb{w}^\top \pmb{x}
\end{equation}
where
\[
w_i = \begin{cases}
0 & \text{ if } i\le(k-1)t \\
( d_{i}+\beta^* -\alpha^*)_+ - (\beta^*-\alpha^*)_+ & \text{ if } (k-1)t+1\le i\le kt \\
(d_{i} -\alpha^*)_+ & \text{ if } i\ge kt+1
\end{cases} 
\]
for all values 
\begin{align*}
(\alpha^*,\beta^*)\in A\times \left\{ 0\right\} & \cup  \left\{ (\alpha,\beta) \ | \ \alpha\in A, \beta = \alpha\right\} \\
 &\cup \left\{ (\alpha,\beta) \ | \  \alpha\in A, \beta = (\alpha - d_{i})_+: \ i=(k-1)t+1,\ldots ,kt\right\},
\end{align*}
where $A:=\left\{ d_{i} \ | \ i=(k-1)t+1 ,\ldots ,n\right\}\cup\left\{ 0\right\}$ and returning the solution of the problem with the smallest optimal value.
\end{theorem}
\begin{proof}
For each given $\pmb{x}\in \X$ we can rewrite the objective value $\max_{\pmb{c}\in \cU_k} \pmb{c}^\top \pmb{x}$ by
\begin{align*}
\hat{\pmb{c}}^\top \pmb{x} + \max_{\pmb{\delta}} \ &  \sum_{i=(k-1)t+1}^{n} \delta_{i} d_{i} x_{i} \\
\text{s.t. } & \sum_{i=(k-1)t+1}^{kt}\delta_{i}\ge 1 \\
& \sum_{i=(k-1)t + 1}^{n} \delta_{i} \le \Gamma \\
& \delta_{i}\in [0,1] \quad i=(k-1)t + 1,\ldots ,n .
\end{align*}
The dual of the above linear program is
\begin{equation}\label{eq:dualMinMax}
\begin{aligned}
\hat{\pmb{c}}^\top \pmb{x} + \min \ & \alpha\Gamma - \beta + \sum_{i=(k-1)t+1}^{n}\gamma_i \\
\text{s.t. } & \alpha - \beta + \gamma_i \ge d_{i}x_{i} \quad i=(k-1)t+1,\ldots ,kt \\
& \alpha + \gamma_i\ge d_{i}x_{i} \quad i=kt+1 ,\ldots ,n\\
& \alpha, \beta \ge 0 \\
& \gamma_i\ge 0 \quad  i=(k-1)t + 1,\ldots ,n .
\end{aligned}
\end{equation}
In each optimal solution of the latter problem for $\gamma_i$ it holds
\[
\gamma_i = ( d_{i}x_{i}+\beta -\alpha)_+ = x_{i}(d_{i}+\beta -\alpha)_+ + (1-x_{i})(\beta-\alpha)_+
\]
for each $i=(k-1)t+1, \ldots ,kt$ and 
\[
\gamma_i = (d_{i}x_{i} -\alpha)_+ = x_{i}(d_{i} -\alpha)_+
\]
for each $i=kt+1, \ldots ,n$. Therefore, substituting the latter equations in the objective function of Problem \eqref{eq:dualMinMax} we obtain that problem
\[
\min_{\pmb{x}\in \X}\max_{\pmb{c}\in \cU_k} \pmb{c}^\top \pmb{x}
\]
is equivalent to
\begin{equation}
\begin{aligned}
\min \ & \hat{\pmb{c}}^\top \pmb{x} + \alpha\Gamma - \beta + \sum_{i=(k-1)t+1}^{kt}x_{i}(d_{i}+\beta -\alpha)_+ + (1-x_{i})(\beta-\alpha)_+ \\
& + \sum_{i=kt+1}^{n} x_{i}(d_{i} -\alpha)_+\\
\text{s.t. } & \pmb{x}\in \X, \alpha,\beta\ge 0.
\end{aligned}
\end{equation}

\begin{figure}[ht!]
\centering
\def\svgwidth{7 cm}
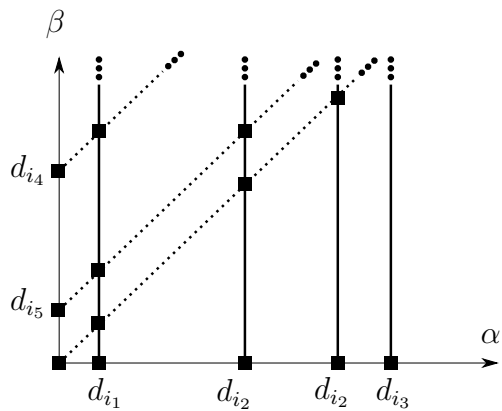
\caption{Non-differentiable regions of $f_1(\alpha,\beta)$ and $f_2(\alpha,\beta)$, respectively drawn as dotted and solid lines. The squares represent the kinkpoints of the objective function on the domain visible in the figure.\label{fig:epigraph}}
\end{figure}

For any fixed $\pmb{x}^*$ the objective function of the latter problem is a piece-wise linear and convex function, so its minimum is reached at one of its kinkpoints. More specifically, the function is the sum of an affine function and two convex piece-wise linear functions, $f_1(\alpha,\beta)=\sum_{i=(k-1)t+1}^{kt}x_{i}(d_{i}+\beta -\alpha)_+ + (1-x_{i})(\beta-\alpha)_+$ and $f_2(\alpha,\beta)=\sum_{i=kt+1}^{n} x_{i}(d_{i} -\alpha)_+$. The non-differentiable regions of both functions are half-lines, parallel to the lines $\alpha-\beta=0$ and $\alpha=0$, respectively, see Figure~\ref{fig:epigraph} for an illustration. Therefore, any kinkpoint of the objective function is obtained at the intersection of these lines, proving the result.

\end{proof}
\begin{corollary}
The heuristic presented in Algorithm~\ref{alg:heuristic} for the partition given in~\eqref{eq:partitionheuristic} requires the solution of $\mathcal O(K(2+t)n)$ many deterministic problems.
\end{corollary}
\begin{proof}
For each of the $K$ subsets $\cU_k$ the number of $(\alpha,\beta)$ values for which we have to solve the deterministic problem is in $\mathcal O (2n+tn)$. Since we have to solve the min-max problem for each of the $K$ subsets in Algorithm \ref{alg:heuristic} in total we have to solve $\mathcal O(K(2+t)n)$ deterministic problems.
\end{proof}

\subsection{Heuristic 2}
In this section we present a second heuristic for Problem~\eqref{eq:minmaxmin}, which is also based on partitioning the set $\cU$ into $K$ sets $\cU_1\cup\ldots\cup\cU_K$, such that the min-max problem on each set can be solved in polynomial time. Differently from the previous approach, we find the partition dynamically.

Consider again the uncertainty set
\[ \Ug = \left\{ \pmb{c}\in\R^n : c_i = \hat{c}_i + d_i \delta_i,\ \pmb{\delta}\in\{0,1\}^n,\ \sum\limits_{i\in[n]} \delta_i \le \Gamma \right\} \]
and let us assume that for a specific item $i\in[n]$, we enforce $\delta_i=1$. We denote the resulting uncertainty set as $\cU^{+i}$. It is also possible to enforce $\delta_i=0$, in which case the resulting set is denoted as $\cU^{-i}$. Note that $\cU^{+i}\cap\cU^{-i}=\emptyset$ and $\cU^{+i}\cup\cU^{-i}=\Ug$. We can repeat this branching step on the resulting subsets, until we have constructed a partition consisting of $K$ sets.

This requires two rules: one rule to decide on which of the current sets to branch, and another rule to decide which variable to fix. This is similar to a branch-and-bound method, where we need to decide a node selection and a variable selection policy. 

We propose the following rules. For branching, we choose a set for which less than $\Gamma$ many items are already fixed to $1$ (otherwise it consists of a single scenario) and the min-max problem has the highest objective value. This is a greedy choice by which we can hope to reduce the objective value of the current solution. For variable selection, we choose an item $i\in[n]$ that is not yet fixed in the current set, is used in the corresponding min-max solution, and has the highest deviation $d_i$. This way, we branch on what is estimated to be the current most important item.

Note that after fixing some $\delta_i$ variables to be either $0$ or $1$, the resulting uncertainty set is a classical budgeted uncertainty set and applying Theorem~3 from~\cite{bertsimas2003combinatorial}, the resulting min-max problem can be solved by solving $\mathcal O(n)$ many deterministic problems. With every branching, we need to solve two new such subproblems. To complete, the heuristic thus requires the solution of $\mathcal{O}(Kn)$ many deterministic problems, and runs overall in polynomial time if the deterministic problem can be solved in polynomial time. Algorithm~\ref{alg:heuristic2} summarizes this procedure.

\begin{algorithm}[htb]
\caption{Heuristic Algorithm for Problem \eqref{eq:minmaxmin} with $K\le n$.}\label{alg:heuristic2}
\algorithmicrequire \ $\X$, $\Ug$, $K$\\
$\mathcal{L}\leftarrow \{\Ug\}$ \\
\Repeat{$|\mathcal{L}|=K$}{
$\cU\leftarrow \argmax \{ \min_{\pmb{x}\in\X} \max_{\pmb{c}\in\cU} \pmb{c}^\top\pmb{x} :  \cU\in \mathcal{L}\text{ not fixed completely} \}$ \label{algchooseU} \\
$\pmb{x} \leftarrow \min_{\pmb{x}\in\X} \max_{\pmb{c}\in\cU} \pmb{c}^\top\pmb{x}$ \\
$i\leftarrow \argmax_{i\in[n]} \{ d_i : x_i=1, \text{ $i$ not yet fixed in $\cU$} \}$ \\
$\mathcal{L} \leftarrow (\mathcal{L} \cup \{ \cU^{+i},\cU^{-i}\}) \setminus \{\cU\}$
}
\algorithmicensure \ $\x{1},\ldots ,\x{K}$ as minimizers for each $\cU\in \mathcal{L}$
\end{algorithm}

We give an example for the approach with $K=3$ in Figure~\ref{fig:h2}.
\begin{figure}[htb]
\begin{center}
\includegraphics[width=0.6\textwidth]{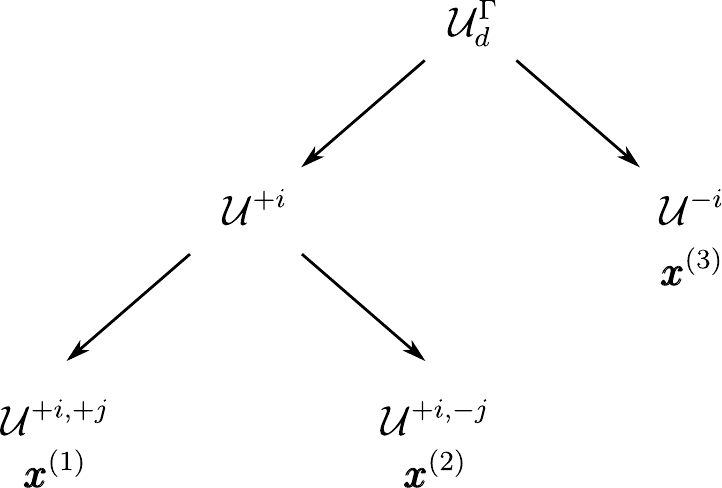}
\caption{Example for Heuristic~2.}\label{fig:h2}
\end{center}
\end{figure}
Here, we first solve the min-max problem using the original set $\Ug$. In the resulting solution, let $i$ be the item with the highest deviation $d_i$. We branch by removing $\Ug$ from our current list of uncertainty sets, and instead consider $\cU^{+i}$ and $\cU^{-i}$. We solve the min-max problem on each. Now let us assume that the resulting objective value is larger on set $\cU^{+i}$. We choose this set for our next branching. Let $j\neq i$ be the item in the corresponding solution with largest deviation value $d_j$. We remove $\cU^{+i}$ from our current partition and instead add $\cU^{+i,+j}$ and $\cU^{+i,-j}$. After solving the respective min-max problems, we have found a partition of $\cU$ into three sets. The algorithm terminates and gives a heuristic solution to Problem~\eqref{eq:minmaxmin} by using an optimal min-max solution on each set of the partition.

Note that if a set $\cU\in\mathcal{L}$ maximizing $\min_{\pmb{x}\in\X} \max_{\pmb{c}\in\cU} \pmb{c}^\top\pmb{x}$ (see Step~\ref{algchooseU} of Algorithm~\ref{alg:heuristic2}) is already completely fixed, it consists of only a single cost scenario and cannot be partitioned further. In this case, the heuristic has in fact determined an optimal solution to Problem~\eqref{eq:minmaxmin}, as the objective value has reached the lower bound \eqref{eq:max-minLB}. Algorithm~\ref{alg:heuristic2} still carries on so that $K$ solutions are found in total.

\section{Computational experiments}
\label{sec:xp}

\subsection{Setup}

In this section we present the results of computational experiments for the minimization variant of the knapsack problem and the shortest path problem. We show results for the exact row-and-column generation method presented in Algorithm \ref{alg:generatescenarios} (RCG), the lower bound \eqref{eq:max-minLB} and both heuristic algorithms (Heur1 and Heur2), presented in Section \ref{sec:heur}. To analyze the quality of the heuristic solutions we compare the results to the classical min-max solution (MM) and to the solution of a heuristic (HeurPS) already presented in \cite{eufinger2017robust}. The latter heuristic calculates $K$ random Pareto-scenarios of the uncertainty set $\Ug$ and returns the optimal deterministic solution for each of the scenarios. 
Note that HeurPS is a more general heuristic that does not exploit the special structure of the budgeted uncertainty set.

All algorithms were implemented in C++. All objective values $cost(\xx)$ were calculated by solving the IP formulation \eqref{eq:IPformulationcosts}. The lower bound \eqref{eq:max-minLB} was calculated by using the row-generation procedure presented in Section \ref{sec:ilp}. The min-max problems appearing in the two heuristics are solved via dualized reformulations (e.g., \eqref{eq:dualMinMax} for the first heuristic) rather than solving the polynomial number of deterministic problems, as the dualized MILP appeared to be faster than the latter approach on our instances. All occurring IP and LP formulations as well as the master- and the slave-problem in Algorithm~\ref{alg:generatescenarios}, were implemented in CPLEX 12.8. For the minimum knapsack problem and the shortest path problem we used the classical IP formulations. To avoid symmetric solutions with the same objective value we added the symmetry-breaking constraints 
\[
\sum_{i=1}^{n} ix^{(j)}_i +1 \le  \sum_{i=1}^{n} i x^{(j+1)}_i \quad j\in[K-1]
\]
to the master-problem \eqref{eq:assignmentformulation}.

The initial ordering of the indices $i_1,\ldots ,i_n$ for Heuristic $1$ was selected by sorting the deviations in non-decreasing order $d_{i_1}\le \ldots \le d_{i_n}$. This choice was motivated by preliminary tests on random instances.

The Pareto-scenarios for the heuristic presented in \cite{eufinger2017robust} were calculated by drawing $K$ random vectors $\pmb{\lambda}_k$ from a uniform distribution and selecting the Pareto-scenarios
\[
\pmb{c}^{k}:=\argmax_{\pmb{c}\in \Ug} \pmb{\lambda}_k^\top \pmb{c} .
\]
We set a timelimit of $3600$ seconds for each of the algorithms. All computations were calculated on a cluster of 64-bit Intel(R) Xeon(R) CPU E5-2603 processors running at 1.60 GHz with 15MB cache. Each algorithm was restricted to one thread.


We consider min-knapsack and shortest path problems. The min-knapsack problem can be written as
$$
\min \left\{\sum_{i\in [n]} c_i x_i : \sum_{i\in [n]} w_i x_i \geq W, \; x\in \{0,1\}^n\right\}.
$$
Our random instances were generated as in \cite{ChasseinGKP19}. For each dimension $n$ the costs $c_i$ and the weights $w_i$ were drawn from a uniform distribution on $\left\{ 1,\ldots ,100\right\}$. The knapsack capacity $W$ was set to $35\%$ of the sum of all weights. For each knapsack instance we generated a budgeted uncertainty set with mean vector $\hat{\pmb{c}}=\pmb{c}$ and a random deviation vector $\pmb{d}$ where each $d_i$ is drawn uniformly in $\left\{ 1,\ldots ,c_i\right\}$. For each dimension $n\in\{100,200,300,400\}$ we generate $10$ random instances and for each instance we vary the parameters $\Gamma\in\{3,6\}$ and $K\in\{10,20,30\}$.

Our shortest path computations were performed on the instances generated in \cite{hanasusanto2015k}. The authors create graphs with $20,25,\ldots ,50$ nodes, corresponding to points in the Euclidean plane with random coordinates in $[0,10]$. They choose a budgeted uncertainty set where $\hat c_{ij}$ is set to the Euclidean distance of node $i$ to node $j$ and the deviations are set to $d_{ij} = \frac{c_{ij}}{2}$. The parameter $\Gamma$ is chosen from $\left\{3,6\right\}$. For each dimension $n$ we tested all $100$ instances generated in \cite{hanasusanto2015k} and for each instance we vary the parameters $\Gamma\in\{3,6\}$ and $K\in\{10,20,30\}$.

\subsection{Results on knapsack problems}

The results regarding the RCG and the MMLB are shown in Table~\ref{tbl:LB_Knapsack}. Each row shows the average over all $10$ knapsack instances of the following values (rounded down to one decimal place):
the number of items $n$; the parameter $\Gamma$; the number of calculated solutions $K$; the percental gap (Gap) between the MMLB and the best LB calculated by the RCG during the timelimit; the total calculation time $t$ in seconds of the MMLB (or RCG); the number of terminated calculations $\#solved$ of the MMLB (or RCG) during the timelimit; the number of iterations $\#iter$ performed by the MMLB (or RCG); the percental optimality gap (Opt-Gap) of the best LB and UB calculated by the RCG during timelimit. Recall that MMLB does not depend on the value of $K$.

\begin{table}[h]
\centering
\resizebox{0.9\textwidth}{!}
{
\begin{tabular}{rrr|r|rrr|rrrr}
&&&&\multicolumn{3}{c}{MMLB}&\multicolumn{4}{|c}{RCG}\\
\hline
$n$ & $\Gamma$ & $K$ & Gap (\%) & $t$ & $\#solved$ & $\#iter$ & $t$ & $\#solved$ & $\#iter$ & Opt-Gap (\%) \\ 
\hline 
100&3&10&8.0&0.4&10&14.0&3600.0&0&2.0&16.9\\
100&3&20&8.0&0.4&10&14.0&3600.0&0&2.0&16.9\\
100&3&30&8.0&0.4&10&14.0&3600.0&0&2.0&16.9\\
100&6&10&13.4&1.9&10&33.6&3600.0&0&2.1&27.7\\
100&6&20&13.4&1.9&10&33.6&3600.0&0&2.0&27.8\\
100&6&30&13.4&1.9&10&33.6&3600.0&0&2.0&27.8 \\
\hline
200&3& - & - &3.1&10&37.6 \\
200&6& - & - &98.2&10&121.9 \\
300&3& - & - &4.1&10&53.8 \\
300&6& - & - &292.5&10&197.4 \\
400&3& - & - &5.8&10&47.0 \\
400&6& - & - &1639.6&9&247.1
\end{tabular}
}
\caption{Results of MMLB and RCG for the knapsack problem.}
\label{tbl:LB_Knapsack}
\end{table}


Even for a dimension of $n=100$ the RCG hit the timelimit of $1$ hour in every instance. Furthermore the lower bound given by the MMLB is at least $8\%$ better than the lower bound of the RCG after $1$ hour. The optimality gap of the RCG after $1$ hour is still at least $17\%$ and sometimes even $27\%$. Due to this observation and the time consuming calculations of the RCG we did not test the RCG for larger instances. The bounds found using MMLB are stronger. For nearly all configurations we could calculate the bound for all instances during the timelimit. The total calculation time is very small for most of the instances. Interestingly the calculation time increases significantly for the larger uncertainty sets with $\Gamma=6$ while for $\Gamma=3$ all instances could be solved in seconds. This is mostly due to the larger number of iterations performed by the MMLB for $\Gamma=6$.

The results for all three heuristics and the min-max solution are shown in Table~\ref{tbl:Heur_Knapsack}.
Each row shows the average over all $10$ knapsack instances of the following values (rounded down to one decimal place):
the number of items $n$; the parameter $\Gamma$; the number of calculated solutions $K$; the percental gap between the MMLB and the solution of Heur1 (or Heur2/HeurPS/MM). We do not record any calculation times here since all procedures return a solution in a few seconds for all instances. The percental gap is always compared to the MMLB since as Table~\ref{tbl:LB_Knapsack} indicates, in nearly all instances this lower bound is tighter than the one provided by the RCG.

\begin{table}[h]
\centering
\resizebox{0.9\textwidth}{!}
{
\begin{tabular}{rrr|rrrr}
$n$ & $\Gamma$ & $K$ & Heur1 & Heur2 & HeurPS & MM \\ 
\hline 
100&3&10&2.5&1.3&7.3&5.4\\
100&3&20&2.2&0.8&6.5&5.4\\
100&3&30&2.5&0.5&6.1&5.4\\
100&6&10&3.7&3.2&9.2&7.0\\
100&6&20&4.0&1.9&8.8&7.0\\
100&6&30&3.6&1.8&8.4&7.0\\
200&3&10&1.9&1.0&4.6&3.4\\
200&3&20&1.6&0.7&4.6&3.4\\
200&3&30&1.7&0.5&4.6&3.4\\
200&6&10&3.4&2.7&7.0&5.2\\
200&6&20&3.2&2.2&7.0&5.2\\
200&6&30&3.1&1.9&7.0&5.2
\end{tabular}
\quad
\begin{tabular}{rrr|rrrr}
$n$ & $\Gamma$ & $K$ & Heur1 & Heur2 & HeurPS & MM \\ 
\hline 
300&3&10&1.5&1.0&3.1&2.7\\
300&3&20&1.4&0.7&3.1&2.7\\
300&3&30&1.3&0.6&3.1&2.7\\
300&6&10&2.7&2.2&5.1&4.2\\
300&6&20&2.5&1.8&5.1&4.2\\
300&6&30&2.4&1.5&5.1&4.2\\
400&3&10&1.0&0.7&2.3&1.9\\
400&3&20&0.9&0.5&2.3&1.9\\
400&3&30&0.7&0.4&2.3&1.9\\
400&6&10&2.1&1.6&3.7&3.0\\
400&6&20&1.7&1.4&3.7&3.0\\
400&6&30&1.6&1.2&3.7&3.0
\end{tabular}
}
\caption{Percental Gaps of Heur1, Heur2, HeurPS and MM to the lower bound MMLB for the knapsack problem.}
\label{tbl:Heur_Knapsack}
\end{table}
Heur2 outperforms the other heuristics for all configurations. The gap to the lower bound is always smaller than $3.2\%$. The gaps of Heur1 are also very small, at most $3.7\%$, but always larger than the gaps of Heur2. The gaps of HeurPS are the largest in most of the instances, even larger than the gaps of the min-max solution. 

In Figure \ref{fig:KnapsackLinePlot} we show a line plot of the same average gaps as in Table \ref{tbl:Heur_Knapsack} over $10$ instances with $n=150$ and $\Gamma=6$ for all $K\in\left\{ 1,\ldots ,n\right\}$. Heur2 shows the best performance. Heur1 returns solutions which are significantly better than the min-max solutions as well. Unfortunately due to the number of items $t:=\lfloor \frac{n}{K}\rfloor$ considered in each of the $K-1$ subsets in the partition constructed in Heur1, the size of the last set in the partition varies  depending on $K$. This explains the fluctuating gaps of Heur1. The gaps of Heur2 seem to be much more stable, as the algorithm guarantees an improving objective value with increasing $K$.
\begin{figure}[htb]
\centering
\includegraphics[scale=0.7]{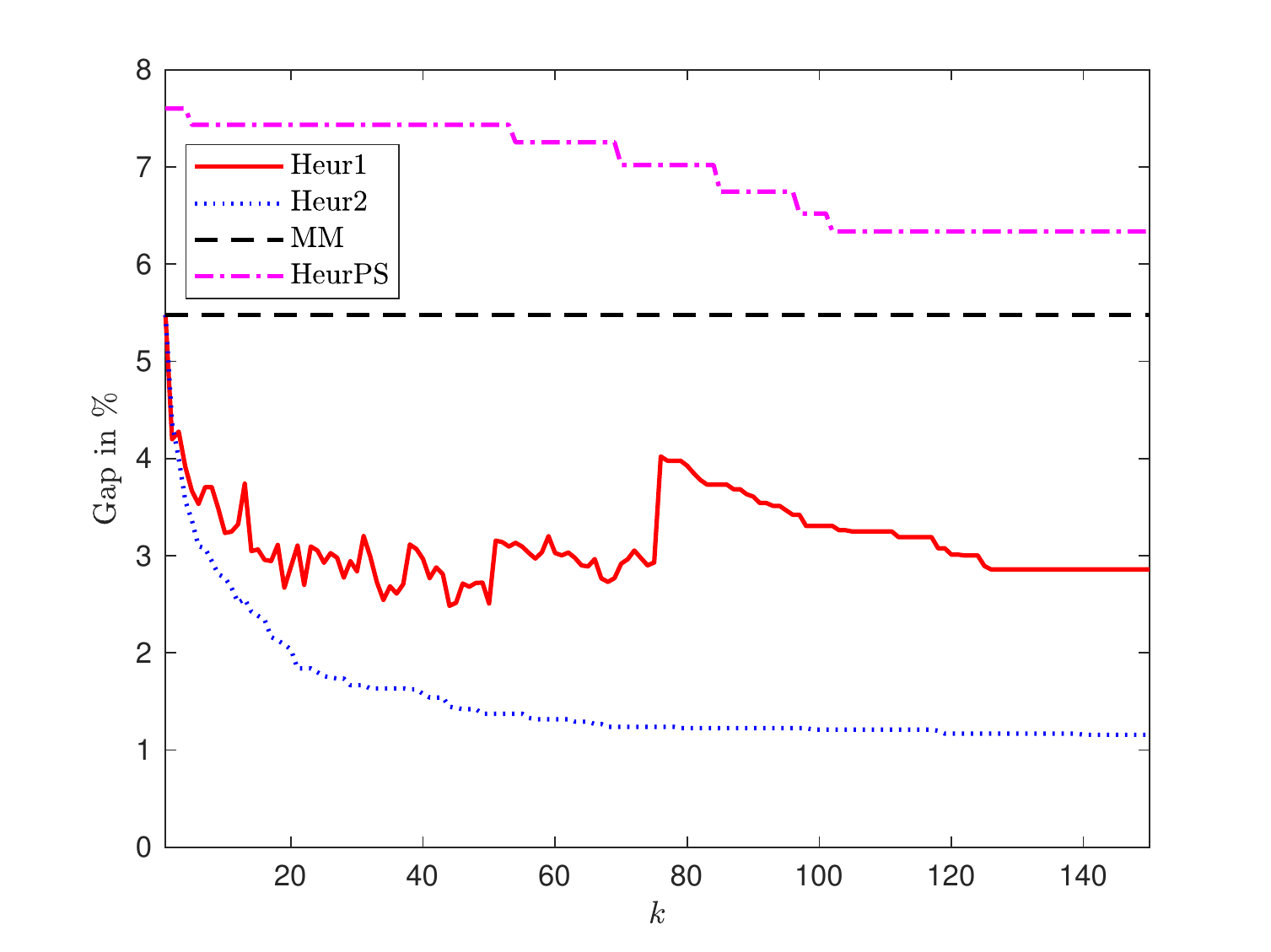}
\caption{Average percental gaps between objective value of the heuristic solutions and the MMLB over $10$ instances with $n=150$ and $\Gamma=6$.}
\label{fig:KnapsackLinePlot}
\end{figure}

\subsection{Results on shortest path problems}

In this section we consider the classical shortest path problem. The results regarding the RCG and the MMLB are shown in Table~\ref{tbl:LB_ShortestPath}. 
The results for all three heuristics and the min-max solution are shown in Table~\ref{tbl:Heur_ShortestPath}. 
Each row shows the average over all  $100$ shortest path instances.

\begin{table}[h]
\centering
\resizebox{0.9\textwidth}{!}
{
\begin{tabular}{rrr|r|rrr|rrrr}
&&&&\multicolumn{3}{c}{MMLB}&\multicolumn{4}{|c}{RCG}\\
\hline
$n$ & $\Gamma$ & $K$ & Gap (\%) & $t$ & $\#solved$ & $\#iter$ & $t$ & $\#solved$ & $\#iter$ & Opt-Gap (\%) \\ 
\hline 
57&3&10&5.4&0.0&100&7.3&3279.9&14&4.8&9.7\\
57&3&20&5.4&0.0&100&7.3&3428.3&10&6.0&9.8\\
57&3&30&5.8&0.0&100&7.3&3360.7&14&5.0&11.0\\
57&6&10&8.4&0.1&100&12.0&3453.1&9&5.0&17.1\\
57&6&20&8.2&0.1&100&12.0&3535.7&7&5.4&16.2\\
57&6&30&9.0&0.1&100&12.0&3572.3&6&5.2&17.7\\
90&3&10&6.4&0.1&100&8.8&3507.4&4&4.3&13.5\\
90&3&20&6.8&0.1&100&8.8&3566.8&2&3.6&15.4\\
90&3&30&8.5&0.1&100&8.8&3567.4&1&3.1&18.8\\
90&6&10&10.4&0.2&100&15.0&3568.9&1&3.9&24.9\\
90&6&20&11.5&0.2&100&15.0&3600.0&0&3.5&26.7\\
90&6&30&12.6&0.2&100&15.0&3600.0&0&3.0&30.0\\
131&3&10&6.6&0.1&100&10.1&3600.0&0&4.1&16.6\\
131&3&20&7.4&0.1&100&10.1&3600.0&0&3.6&18.5\\
131&3&30&8.8&0.1&100&10.1&3600.0&0&3.0&22.5\\
131&6&10&11.4&0.3&100&18.3&3600.0&0&3.6&30.9\\
131&6&20&12.5&0.3&100&18.3&3600.0&0&3.3&32.1\\
131&6&30&13.5&0.3&100&18.3&3600.0&0&2.9&34.2 \\
\hline
179&3&-&-&0.1&100&12.0 \\
179&6&-&-&0.6&100&22.9 \\
234&3&-&-&0.2&100&11.8 \\
234&6&-&-&0.6&100&23.7 \\
297&3&-&-&0.2&100&12.5 \\
297&6&-&-&0.9&100&28.2 \\
368&3&-&-&0.3&100&13.7 \\
368&6&-&-&1.2&100&32.3 \\
\end{tabular}
}
\caption{Results of MMLB and RCG for the shortest path problem.}
\label{tbl:LB_ShortestPath}
\end{table}


Even for a dimension of $n=131$ the RCG hit the timelimit of $1$ hour in every instance. For most configurations it could solve at most $25$ of $100$ instances during the timelimit. Furthermore the lower bound given by the MMLB is larger than the lower bound of the RCG after $1$ hour. For small instances it is at least $5\%$ better while for the larger instances the gap increases up to $12\%$ for some instances. The optimality gap of the RCG after $1$ hour is still at least $9\%$, for larger instances even around $30\%$. Due to this observation and the time consuming calculations of the RCG we did not test the RCG for larger instances as the bounds provided by MMLB are tighter and the hardest of them could be computed in at most $1.2$ seconds. This is due to the very small number of iterations and the small computational effort of the shortest path problem in its deterministic version. 

\begin{table}[h]
\centering
\resizebox{0.9\textwidth}{!}
{
\begin{tabular}{rrr|rrrr}
$n$ & $\Gamma$ & $K$ & Heur1 & Heur2 & HeurPS & MM \\ 
\hline 
57&3&10&5.1&1.7&8.1&15.3\\
57&3&20&5.6&0.5&6.1&15.3\\
57&3&30&7.4&0.3&5.5&15.3\\
57&6&10&8.2&3.9&8.7&17.4\\
57&6&20&7.9&2.1&6.5&17.4\\
57&6&30&9.1&1.7&5.8&17.4\\
90&3&10&5.1&2.3&11.8&18.0\\
90&3&20&5.1&0.9&9.5&18.0\\
90&3&30&4.6&0.5&8.6&18.0\\
90&6&10&9.5&5.7&14.1&22.6\\
90&6&20&8.6&3.3&11.2&22.6\\
90&6&30&8.0&2.4&10.2&22.6\\
131&3&10&6.4&3.4&15.7&19.9\\
131&3&20&5.8&1.7&13.6&19.9\\
131&3&30&5.7&0.9&13.0&19.9\\
131&6&10&10.9&7.5&20.4&26.1\\
131&6&20&8.9&4.8&18.0&26.1\\
131&6&30&9.4&3.6&16.5&26.1
\end{tabular}
\quad
\begin{tabular}{rrr|rrrr}
$n$ & $\Gamma$ & $K$ & Heur1 & Heur2 & HeurPS & MM \\ 
\hline 
179&3&10&7.3&5.2&20.2&22.4\\
179&3&20&7.1&3.2&17.7&22.4\\
179&3&30&7.9&2.2&15.8&22.4\\
179&6&10&11.4&8.9&23.1&29.6\\
179&6&20&10.5&5.8&20.2&29.6\\
179&6&30&10.7&4.7&18.7&29.6\\
234&3&10&7.3&5.7&22.5&22.7\\
234&3&20&6.8&3.3&20.1&22.7\\
234&3&30&6.9&2.1&19.1&22.7\\
234&6&10&11.7&9.8&27.5&31.0\\
234&6&20&11.0&6.5&25.0&31.0\\
234&6&30&10.0&5.1&23.5&31.0\\
297&3&10&7.7&6.5&24.1&22.9\\
297&3&20&7.3&4.5&20.9&22.9\\
297&3&30&7.7&3.1&20.0&22.9\\
297&6&10&12.3&10.4&29.9&32.3\\
297&6&20&10.9&7.5&28.1&32.3\\
297&6&30&10.6&5.9&26.8&32.3\\
368&3&10&7.5&7.3&25.5&23.3\\
368&3&20&6.9&5.2&23.7&23.3\\
368&3&30&6.7&3.5&22.8&23.3\\
368&6&10&13.3&12.5&32.5&33.7\\
368&6&20&11.1&9.1&30.9&33.7\\
368&6&30&10.8&7.5&29.7&33.7
\end{tabular}
}
\caption{Percental Gaps of Heur1, Heur2, HeurPS and MM to the lower bound MMLB for the shortest path problem.}
\label{tbl:Heur_ShortestPath}
\end{table}
In Table \ref{tbl:Heur_ShortestPath} we show the results for all three heuristics and the min-max solution. Heur2 outperforms all other heuristics for all configurations. Compared to the knapsack problem the gaps are slightly larger for higher dimensions but are never larger than $12\%$. The gaps of Heur1 are also larger than for the knapsack problem, at most $13.3\%$, but always larger than the gaps of Heur2. In contrast to the knapsack problem here the min-max solution provides the worst gaps in nearly all instances. The gaps of HeurPS are slightly better but can also increase up to $30\%$ for larger instances. To summarize Heur1 and Heur2 seem to be a good choice to solve Problem \eqref{eq:minmaxmin} for the shortest path problem.

In Figure \ref{fig:ShortestPathLinePlot} we show a line plot of the same average gaps as in Table \ref{tbl:Heur_ShortestPath} over $100$ instances with $n=179$ and $\Gamma=6$ for all $K\in\left\{ 1,\ldots ,n\right\}$. Again, Heur2 outperforms the other heuristics. Heur1 returns solutions which are significantly better than the min-max solutions as well. 
In contrast to the knapsack problem the HeurPS performs better than the min-max solution here.
\begin{figure}[h]
\centering
\includegraphics[scale=0.7]{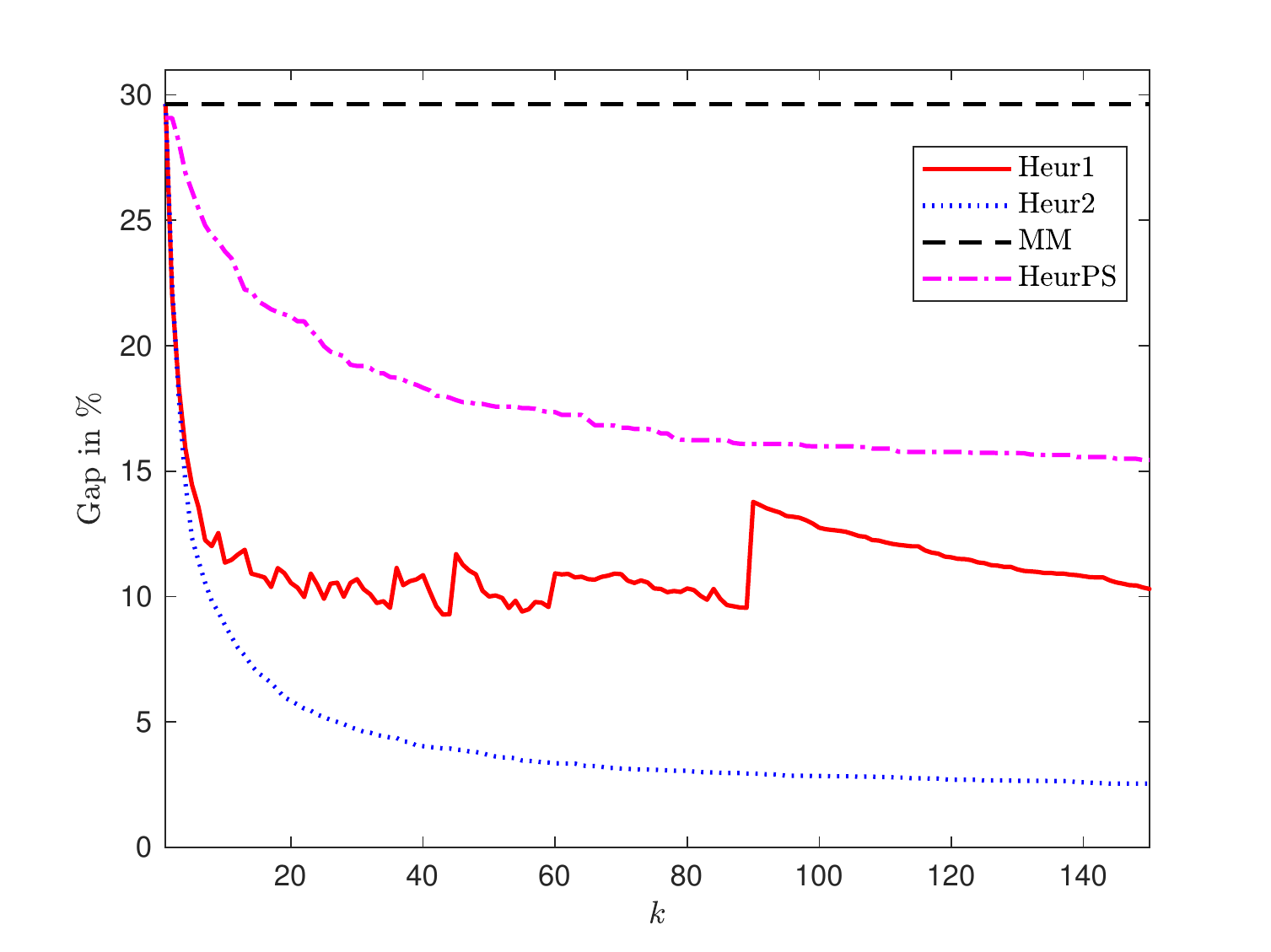}
\caption{Average percental gaps between objective value of the heuristic solutions and the MMLB over $100$ instances with $n=179$ and $\Gamma=6$.}
\label{fig:ShortestPathLinePlot}
\end{figure} 

\section{Conclusion}
\label{sec:conclusion}

We considered the min-max-min problem in robust combinatorial optimization, where it is possible to prepare $K$ solutions beforehand. Once the uncertain costs are revealed, one then chooses the best of the prepared solutions for this scenario. For the first time, the min-max-min setting is considered in combination with discrete budgeted uncertainty.

Our complexity analysis reveals that most combinatorial problems become $\NP$-hard in this setting, and even inapproximable. Furthermore, even evaluating the objective value of a $K$-tuple of solutions is already $\NP$-hard, making it unlikely that a compact problem formulation exists. We thus present a row-and-column generation approach to find exact solutions. As this approach fails for larger problem instances, we also develop two heuristic algorithms that run in polynomial time. Computational experiments indicate that these heuristics scale well with the problem size, leading to solutions in seconds that leave a gap of a few percent for large instances when compared to a simple lower bound.

\section*{References}

\bibliographystyle{alpha}
\bibliography{references}

\newpage
\appendix

\section{Additional proofs}
\label{app:proofs}

\begin{theorem}
Problem \eqref{eq:minmaxmin} for the spanning tree problem is weakly $\NPhard$, 
even if $K=2$.
It is strongly $\NPhard$ if $K$ is part of the input.
\end{theorem}
\begin{proof}
Assume first that $\Gamma=1$ and $K=2$. We reduce the $2$-partition problem to \eqref{eq:minmaxmin} for the spanning tree problem. Given an instance of the $2$-partition problem i.e. $a_i\in\N$ for each $i\in N=[n]$ we define the graph $G=(V,E)$ as follows: Let $V=\left\{ v_1,\ldots ,v_{n+1}, w_1,\ldots , w_{n+1}\right\}$ and $E=\left\{e_1,\ldots ,e_n, f_1, \ldots ,f_n , g_1,\ldots , g_{n+1} \right\}$ where $e_i = \left\{ v_i, v_{i+1}\right\}$, $f_i = \left\{ w_i, w_{i+1}\right\}$ and $g_i = \left\{ v_i, w_i\right\}$. Assume $e_i$ and $f_i$ have nominal costs $a_i$ and a deviation of $M$ where $M=\sum_{i \in [n]} a_i$. All edges $g_i$ have costs and deviation $0$. Note that all optimal spanning trees of the latter graph must use all edges $g_i$ and for each $i\in N$ exactly one of the edges $e_i$ or $f_i$. Now any optimal solution of \eqref{eq:minmaxmin} for $K=2$ contains two trees which are disjoint on the $e$-edges and on the $f$-edges since otherwise the deviation on a common edge could be set to $M$. Thus we can find a solution $I\subseteq N$ for the $2$-partition problem if and only if the optimal value of \eqref{eq:minmaxmin} is $\frac{1}{2}M$.

The proof extends to the case when $K$ is part of the input by setting $\Gamma=K-1$ and constructing a graph $G=(V,E)$ with nodes
\[V=\left\{ v^1_1,\ldots ,v^1_{n+1},\ldots, v^K_1,\ldots ,v^K_{n+1}\right\}\]
and edges
\[E=\left\{e^1_1,\ldots ,e^1_n,\ldots ,e^K_1,\ldots ,e^K_n, g^1_1,\ldots , g^1_{n+1},\ldots ,g^{K-1}_1,\ldots , g^{K-1}_{n+1} \right\}\]
where $e^j_i = \left\{ v^j_i, v^j_{i+1}\right\}$ and $g^j_i = \left\{ v^j_i, v^{j+1}_i\right\}$. Again all edges $e^j_i$ have nominal costs $a_i$ and a deviation of $M$ where $M=\sum_{i \in [n]} a_i$. All edges $g^j_i$ have costs and deviation $0$. A similar reasoning as above shows that one can decide if a partition into $K$ sets  having the same costs exists if and only if the corresponding min-max-min problem has an optimal value equal to $\frac{1}{K}M$. The problem generalizes the 3-partition problem, which is $\NPhard$ in the strong sense, which proves the result.
\end{proof}


\begin{theorem}
Problem \eqref{eq:minmaxmin} for the assignment problem is weakly $\NPhard$,
even if $K=2$.
It is strongly $\NPhard$ if $K$ is part of the input.
\end{theorem}
\begin{proof}
Assume first that $\Gamma=1$ and $K=2$. We reduce the $2$-partition problem to \eqref{eq:minmaxmin} for the assignment problem. Given an instance of the $2$-partition problem i.e. $a_i\in\N$ for each $i\in N=[n]$ we want to know if there exists a subset $I\subseteq N$ with $|I| = |N\setminus I|$ such that $\sum_{i\in I} a_i = \sum_{i\in N\setminus I} a_i$. We consider a graph $G=(V,E)$ with nodes $V=\left\{ v_1,\ldots ,v_{n}, w_1,\ldots , w_{n}\right\}$ and edges $E=\left\{\left\{v_i,w_j\right\}: i,j=1,\ldots ,n \right\}$. The edges $\left\{v_i,w_i\right\}$ have nominal costs $-a_i$ and deviation $M=\sum_{i \in [n]} a_i$ for each $i=1,\ldots ,n$. All other edges have costs and deviation $0$. By the choice of $M$ the two solutions in an optimal solution of \eqref{eq:minmaxmin} are disjoint and each edge $\left\{v_i,w_i\right\}$ is used by at least one of the two solutions. Thus we can find a solution $I\subseteq N$ for the $2$-partition problem with $|I| = |N\setminus I|$ if and only if the optimal value of \eqref{eq:minmaxmin} is $-\frac{1}{2}M$. 

The proof extends to the case when $K$ is part of the input by the same construction and $\Gamma=K-1$. A similar reasoning as above shows that one can decide if a partition into $K$ sets  having the same costs exists if and only if the corresponding min-max-min problem has an optimal value equal to $\frac{1}{K}M$. The problem generalizes the 3-partition problem, which is $\NPhard$ in the strong sense, which proves the result.

\end{proof}

\section{Dynamic programming for the knapsack problem with any fixed $K$ and $\Gamma$}
\label{app:DP}

Consider first $\Gamma\ge 2$ and $K=2$. As before, the algorithm enumerates labels $\pmb{s}$ and chooses the best of them by computing their costs. As $K=2$, only two solutions are being built, and steps~\ref{step:x1}--\ref{step:x2} follow the same idea as before with one difference: computing the cost \eqref{eq:cost} requires the worst $\Gamma$ deviations for each partial solution. Therefore, every state $s\in\S$ is now described by the $(3\Gamma+4)$-tuple $\pmb{s}=(w^{(1)},w^{(2)},c^{(1)},c^{(2)},\boldi^{(1)},\boldi^{(2)}, \boldi^{(1,2)})$ where $\boldi^{(1)}$, $\boldi^{(2)}$, and $\boldi^{(1,2)}$ are $\Gamma$-tuples recording the indices of the largest elements. Equation~\eqref{eq:cost} becomes
\begin{equation}
\label{eq:generalcost}
cost(\pmb{s}) = \max_{S\subseteq [n] \atop |S| \leq \Gamma}\left\{\min\left(c^{(1)}+\sum_{i\in S\cap (\boldi^{(1)}\cup \boldi^{(1,2)})}d_i,c^{(2)}+\sum_{i\in S\cap (\boldi^{(2)}\cup\boldi^{(1,2)})}d_i\right)\right\}.
\end{equation}

For $K\geq 3$, we are now constructing $K$ solutions, so steps~\ref{step:x1}--\ref{step:x2} should be adapted accordingly. In addition, computing the cost for $K\ge 3$ also requires the $\Gamma$ worst deviations for each subset $\boldj\subseteq[K]$ that contains at most $\Gamma$ elements. We obtain states described by
$$
\pmb{s}=(w^{(1)},\ldots,w^{(K)},c^{(1)},\ldots,c^{(K)},\boldi^{(\boldj_1)},\ldots,\boldi^{(\boldj_{\mathfrak{K}})}),
$$
where 
$$
\mathfrak{K}=\sum_{\gamma=1}^{\min(K,\Gamma)}{K \choose \gamma}, 
$$
which is constant when $K$ and $\Gamma$ are constant. The resulting set $\S$ contains $O(n^{K+\Gamma \mathfrak{K}}\maxc^K)$ many states. The cost function~\eqref{eq:generalcost} can be extended similarly.
\end{document}

%% file: newcommands.tex
\newcommand{\maxw}{\overline{w}}
\newcommand{\maxc}{\overline{c}}
\newcommand{\maxd}{\overline{d}}

\newcommand{\boldi}{\mathbf{i}}
\newcommand{\boldj}{{\mathbf{j}}}

\newcommand{\hc}{\hat{c}}
\newcommand{\N}{\mathbb{N}}

\newcommand{\cU}{\mathcal{U}}
\newcommand{\U}{{\mathcal{U}}}
\newcommand{\Ug}{{\mathcal{U}}^\Gamma}
\newcommand{\Ud}{{\mathcal{U}}}

\newcommand{\R}{\mathbb{R}}
\newcommand{\X}{\mathcal{X}}

\renewcommand{\S}{\mathcal{S}}

\newcommand{\xx}{\mathbf{x}}

\newcommand{\x}[1]{\pmb{x}^{(#1)}}

\newcommand{\binvar}[1]{\{ 0,1 \}^{#1}}

\renewcommand{\algorithmicrequire}{\textbf{Input:}}
\renewcommand{\algorithmicensure}{\textbf{Output:}}
\newcommand{\NPhard}{\mathcal{NP}\mbox{-}\text{hard}}

\newcommand{\NP}{\mathcal{NP}}
\renewcommand{\P}{\mathcal{P}}

\renewcommand{\S}{\mathcal{S}}

\DeclareMathOperator*{\argmin}{arg\,min} 
\DeclareMathOperator*{\argmax}{arg\,max}

%% file: epigraph.pdf_tex
\begingroup%
  \makeatletter%
  \providecommand\color[2][]{%
    \errmessage{(Inkscape) Color is used for the text in Inkscape, but the package 'color.sty' is not loaded}%
    \renewcommand\color[2][]{}%
  }%
  \providecommand\transparent[1]{%
    \errmessage{(Inkscape) Transparency is used (non-zero) for the text in Inkscape, but the package 'transparent.sty' is not loaded}%
    \renewcommand\transparent[1]{}%
  }%
  \providecommand\rotatebox[2]{#2}%
  \newcommand*\fsize{\dimexpr\f@size pt\relax}%
  \newcommand*\lineheight[1]{\fontsize{\fsize}{#1\fsize}\selectfont}%
  \ifx\svgwidth\undefined%
    \setlength{\unitlength}{297.63779528bp}%
    \ifx\svgscale\undefined%
      \relax%
    \else%
      \setlength{\unitlength}{\unitlength * \real{\svgscale}}%
    \fi%
  \else%
    \setlength{\unitlength}{\svgwidth}%
  \fi%
  \global\let\svgwidth\undefined%
  \global\let\svgscale\undefined%
  \makeatother%
  \begin{picture}(1,0.76190476)%
    \lineheight{1}%
    \setlength\tabcolsep{0pt}%
    \put(0,0){\includegraphics[width=\unitlength,page=1]{epigraph.pdf}}%
    \put(0.89166376,0.12599226){\color[rgb]{0,0,0}\makebox(0,0)[lt]{\lineheight{1.25}\smash{\begin{tabular}[t]{l}$\alpha$\end{tabular}}}}%
    \put(0.06551596,0.71815462){\color[rgb]{0,0,0}\makebox(0,0)[lt]{\lineheight{1.25}\smash{\begin{tabular}[t]{l}$\beta$\end{tabular}}}}%
    \put(0,0){\includegraphics[width=\unitlength,page=2]{epigraph.pdf}}%
    \put(0.14471089,0.01799898){\color[rgb]{0,0,0}\makebox(0,0)[lt]{\lineheight{1.25}\smash{\begin{tabular}[t]{l}$d_{i_1}$\end{tabular}}}}%
    \put(0.38949545,0.01619931){\color[rgb]{0,0,0}\makebox(0,0)[lt]{\lineheight{1.25}\smash{\begin{tabular}[t]{l}$d_{i_2}$\end{tabular}}}}%
    \put(0.57488382,0.01979864){\color[rgb]{0,0,0}\makebox(0,0)[lt]{\lineheight{1.25}\smash{\begin{tabular}[t]{l}$d_{i_2}$\end{tabular}}}}%
    \put(0.69187637,0.01799898){\color[rgb]{0,0,0}\makebox(0,0)[lt]{\lineheight{1.25}\smash{\begin{tabular}[t]{l}$d_{i_3}$\end{tabular}}}}%
    \put(0,0){\includegraphics[width=\unitlength,page=3]{epigraph.pdf}}%
    \put(-0.00251978,0.44097217){\color[rgb]{0,0,0}\makebox(0,0)[lt]{\lineheight{1.25}\smash{\begin{tabular}[t]{l}$d_{i_4}$\end{tabular}}}}%
    \put(-0.00251978,0.1889881){\color[rgb]{0,0,0}\makebox(0,0)[lt]{\lineheight{1.25}\smash{\begin{tabular}[t]{l}$d_{i_5}$\end{tabular}}}}%
    \put(0,0){\includegraphics[width=\unitlength,page=4]{epigraph.pdf}}%
  \end{picture}%
\endgroup%